\documentclass[11pt,reqno]{amsproc}

\title[Dissipation for Muskat]{Note on the dissipation for the general Muskat problem}

\author{Susanna V. Haziot}
\address{Department of Mathematics, Brown University, Providence, RI 02912}
\email{susanna\_haziot@brown.edu}

\author{Beno\^it Pausader}
\address{Department of Mathematics, Brown University, Providence, RI 02912}
\email{benoit\_pausader@brown.edu}

\usepackage[margin=1in]{geometry}
\usepackage{amsmath, amsthm, amssymb, mathrsfs, stmaryrd}

\theoremstyle{plain}
\newtheorem{theorem}{Theorem}[section]

\newtheorem{lemma}[theorem]{Lemma}
\theoremstyle{definition}
\newtheorem{remark}[theorem]{Remark}

\def\lb{\llbracket}
\def\rb{\rrbracket}

\numberwithin{equation}{section}

\date{today}
\begin{document}
\begin{abstract}
We consider the dissipation of the Muskat problem and we give an elementary proof of a surprising inequality of Constantin-Cordoba-Gancedo-Strain \cite{CCGS,CCGS2} which holds in greater generality.
\end{abstract}


\noindent

\maketitle

\section{Introduction}

\subsection{The Muskat problem} The general Muskat problem describes the dynamics of two immiscible fluids  in a porous medium with different densities $\rho^\pm$ and viscosities $\mu^\pm$. Let us denote the interface between the two fluids by $\Sigma$ and assume that it is the graph of a time-dependent function $\eta(x, t)$, i.e. 
\begin{equation*}\label{Sigma}
\Sigma_t=\{(x, \eta(x,t)): x\in \mathbb{R}^d\}.
\end{equation*}
The associated time-dependent fluid domains are then given by
\begin{equation*}\label{Omega+}
\begin{split}
\Omega^+_t&=\{(x, y)\in \mathbb{R}^d\times \mathbb{R}: \eta(x,t)<y<\underline{b}^+(x)\},\\
\Omega^-_t&=\{(x, y)\in \mathbb{R}^d\times \mathbb{R}: \underline{b}^-(x)<y<\eta(x,t)\}
\end{split}
\end{equation*}
where $\underline{b}^\pm$ are the parametrizations of the rigid boundaries
\begin{equation*}\label{Gamma:pm}
\Gamma^\pm =\{(x, \underline{b}^\pm(x)): x\in \mathbb{R}^d\}.
\end{equation*}

The incompressible fluid velocity $u^\pm$ in each region is governed by Darcy's law:
\begin{equation*}\label{Darcy:pm}
\mu^\pm u^\pm+\nabla_{x, y}p^\pm=-(0, \rho^\pm),\quad \hbox{div}_{x, y} u^\pm =0\qquad \text{in}~\Omega^\pm_t.
\end{equation*}

At the interface $\Sigma_t$, the normal velocity is continuous, the jump\footnote{here and in the following, we denote the jump of a quantity at the interface to be
\begin{equation*}
\lb f\rb=f^--f^+.
\end{equation*}}
 of the pressure at the interface is related to the surface tension coefficient $\sigma\ge0$ and the interface moves with the fluid: for $\nu=\frac{1}{\sqrt{1+|\nabla \eta|^2}}(-\nabla \eta, 1)$ the upward pointing unit normal to $\Sigma_t$,
\begin{equation*}\label{u.n:pm}
u^+\cdot \nu=u^-\cdot \nu,\quad\lb p\rb =\sigma \kappa(x)\quad\text{on}~\Sigma_t,\qquad \partial_t\eta=\sqrt{1+|\nabla \eta|^2}u^-\cdot \nu\vert_{\Sigma_t},
\end{equation*}
where the mean curvature is given by
\begin{equation*}
\kappa(x):=-\hbox{div}\left(\frac{\nabla\eta}{\sqrt{1+\vert\nabla\eta\vert^2}}\right).
\end{equation*}

Finally, at the two rigid boundaries, the no-penetration boundary conditions are imposed:
\begin{equation*}\label{bc:b:pm}
u^\pm\cdot \nu^\pm=0\quad \text{on}~\Gamma^\pm,
\end{equation*}
where $\nu^\pm=\pm\frac{1}{\sqrt{1+|\nabla \underline b^\pm|^2}}(-\nabla \underline b^\pm, 1)$ denotes the outward pointing unit normal to $\Gamma^\pm$. In case one or both of $\Gamma^\pm$ is empty (infinite depth), \eqref{bc:b:pm} is  then replaced by the vanishing of $\nabla u$ at infinity. 

As in \cite{FlNg,NgPa}, we shall refer to the system \eqref{Omega+}-\eqref{bc:b:pm} as the two-phase Muskat problem. When the top phase corresponds to vacuum, i.e. $\mu^+=\rho^+=0$, the two-phase Muskat problem reduces to the one-phase Muskat problem and \eqref{u.n:pm} becomes
\begin{equation*}
p^-=\sigma \kappa(x)\quad\text{on}~\Sigma_t,\qquad \partial_t\eta=\sqrt{1+|\nabla \eta|^2}u^-\cdot \nu\vert_{\Sigma_t}.
\end{equation*}

In the following, we will make the following assumptions which are enough to use the trace theory developed in \cite{LeoTice} (see also \cite[Section 3 and Appendix A]{NgPa} for the notations and slight expansions):
\begin{enumerate}
\item $({\bf H}1)$: either $\Gamma^\pm=\emptyset$ or $\underline{b}^\pm\in \dot{W}^{1,\infty}(\mathbb{R}^d)$ and $\hbox{dist}(\Gamma^\pm,\Sigma)>0$.
\item $({\bf H}2)$: $\eta\in L^1\cap\dot{H}^1(\mathbb{R}^d)$. In particular, $\eta\in H^1\subset \widetilde{H}^\frac{1}{2}_\Theta$, the space of trace functions used in \cite{NgPa}.
\end{enumerate}

\subsection{A variational formulation for the velocities}

Starting from \eqref{Darcy:pm} and taking a divergence we get, with the incompressibility condition in \eqref{Darcy:pm}, that $u$ is related to the gradient of a harmonic function:
\begin{equation}\label{HarmonicProblem}
\begin{split}
u^\pm=-\frac{1}{\mu^\pm}\nabla q^\pm,\qquad\Delta q^\pm=0,\qquad q^\pm=p+\rho^\pm y,\qquad\text{in}~\Omega^\pm,\\\frac{1}{\mu^+}\partial_\nu q^+=\frac{1}{\mu^-}\partial_\nu q^-,\qquad
\lb q^\pm\rb= \sigma \kappa(x)+\lb \rho\rb\eta(x),\qquad\text{ on }~\Sigma_t,\\\partial_\nu q^\pm=0\qquad\text{ on }\Gamma^\pm.
\end{split}
\end{equation}

We can introduce the kinetic energy
\begin{equation*}
\begin{split}
\mathcal{E}(t):=\mu\int_{\Omega}\vert u(x,t)\vert^2dx=\mu^+\int_{\Omega^+}\vert u^+(x,t)\vert^2dx+\mu^-\int_{\Omega^-}\vert u^-(x,t)\vert^2dx
\end{split}
\end{equation*}
and using \eqref{HarmonicProblem} and \eqref{u.n:pm}, we also obtain the energy dissipation equality
\begin{equation*}
\begin{split}
\frac{d}{dt}\left(\frac{\lb\rho\rb}{2}\int_{\mathbb{R}^{d}}\eta^2(x,t)dx+\sigma\,\hbox{rArea}(\Sigma_t)\right)=-\mathcal{E}(t),
\end{split}
\end{equation*}
where we define the renormalized area to be
\begin{equation*}
\hbox{rArea}(\Sigma_t):=\int_{\mathbb{R}^d}\left\{\sqrt{1+\vert\nabla\eta(x)\vert^2}-1\right\}dx.
\end{equation*}
The monotonicity of the $L^2$ norm is a classical important result for the Muskat equation \cite{CCGS} (see also \cite{Alazard-Meunier-Smets:Lyapunov-Hele-Shaw-2020} for a list of other monotonous quantities). Using e.g. \cite{NgPa} and the notations therein, without surface tension, in infinite depth, and assuming a Rayleigh-Taylor condition, the energy dissipation can be recasted in the form
\begin{equation*}
\begin{split}
-\frac{d}{dt}\frac{\lb \rho\rb}{2}\int_{\mathbb{R}^{d}}\eta^2(x,t)dx=\frac{1}{\mu^-}\int_{\mathbb{R}^{d}}\eta\cdot G^-(\eta)f^- dx&= \frac{1}{\mu^++\mu^-}\Vert \vert\nabla\vert^\frac{1}{2}\eta\Vert_{L^2}^2+O(\eta^3),\\
\vert O(\eta^3)\vert&\lesssim F(\Vert\eta\Vert_{H^{2d+1}})\Vert \nabla\eta\Vert_{L^\infty}\Vert \eta\Vert_{H^\frac{1}{2}}^2
\end{split}
\end{equation*}
which would suggest that the dissipation rate for the Muskat equation may control derivatives of the solutions and be a useful quantity in the analysis. However, the surprising result of \cite[Section 2]{CCGS} (later expanded to $3d$ in \cite{CCGS2}) shows that, at least in certain settings (no surface tension, absence of boundary, equal viscosities), a much simpler (and weaker) lower bound, in fact, holds:
\begin{equation}\label{SurprisingUpperBound}
\begin{split}
0\le\mathcal{E}(t)\le C \Vert \eta\Vert_{L^1(\mathbb{R}^d)}
\end{split}
\end{equation}
The purpose of this note is to give an elementary, variational, proof of \eqref{SurprisingUpperBound} which extends this inequality to various settings.

\begin{theorem}\label{MainThm}

Assume $({\bf H}1)$ and $({\bf H}2)$, $\sigma \ge0$, $\lb \rho\rb>0$ and $\min\{\mu^-,\mu^+\}>0$, then there holds that
\begin{equation}\label{AlternativeBound}
\begin{split}
0\le \mathcal{E}(t)<\lb\rho\rb^2\left[\mu^+\Vert \eta_-(t)\Vert_{L^1(\mathbb{R}^d)}+\mu^-\Vert \eta_+(t)\Vert_{L^1(\mathbb{R}^d)}\right].
\end{split}
\end{equation}
Moreover, the constants are optimal.

\end{theorem}

In the inequality above and in the rest of the paper, we define the positive and negative part of a function with lower subscripts: $\eta_-:=\min\{0,\eta\}$, $\eta_+:=\max\{\eta,0\}$.

\begin{remark}
\begin{enumerate}
\item Although the constants in \eqref{AlternativeBound} are optimal, the inequality is never saturated, unless $y=0$.
\item A simple variation of the analysis allows to consider piecewise smooth interfaces and thus allows corners, which is significant in view of \cite{GGNP}.
\item Simple variations allow to consider other domains than $\mathbb{R}^d$, such as $\mathbb{T}^d$.
\end{enumerate}
\end{remark}

We finish with some open questions:
\begin{enumerate}
\item It would be interesting to understand the minimal analytic setting under which Theorem \ref{MainThm} holds. In particular, whether in the infinite depth setting, one can extend the result to $\eta\in L^1\cap \dot{H}^\frac{3}{2}$, which seems to be optimal in view of recent advances in local well-posedness theory in \cite{Alazard-Nguyen:Endpoint-Sobolev-Muskat-2020,Alazard-Nguyen:Muskat-Critical-II-2021}. In case $\Gamma^\pm\ne\emptyset$, it would be interesting to know the largest space allowable for $\eta$ and $\underline{b}^\pm$ in order for the boundary problem to be well defined.
\item In case $x\in\mathbb{R}$ and $\eta$ has different limits at $\pm\infty$ (as e.g. \cite{Deng-Lei-Lin:2d-muskat-monotone-data-2017}), it seems that the Dirichlet energy is unbounded for infinite bottom, but can be finite in case $\Vert \underline{b}^+-\underline{b}^-\Vert_{L^\infty}<\infty$. In this case, it would be interesting to investigate whether there is an analogue of Theorem \ref{MainThm} involving some renormalization of $\eta$. 
\item In a similar spirit, it would be interesting to extend the above results to the case when the interface is not a graph, especially since such interfaces can dynamically form \cite{CCFG}, and whether there is a connection with the problematic local well-posedness theory.
\item It would be great to clarify what is true and what is not in the one-fluid case, and whether the two-fluid problem is ``more (or less) stable'' in a sense to be made precise. However, we refer to \cite{Ng} which may suggest that stronger bounds than \eqref{SurprisingUpperBound} hold in the one-fluid case.

\end{enumerate}

\section{A minimization problem}

\subsection{The minimization problem}

All of our considerations are instantaneous; from now on, we will fix an interface $\eta\in W^{1,\infty}(\mathbb{R}^d)$.

We would like to express $\mathcal{E}$ as the solution of the following minimization problem:
\begin{equation}\label{MinPb}
\begin{split}
\mathfrak{m}:=\min\{\mathbb{E}[f^+,f^-],\,\, (f^+,f^-)\in\mathcal{A}\}
\end{split}
\end{equation}
where the set of admissible pairs is given by
\begin{equation*}
\begin{split}
\mathcal{A}:=\{(f^+,f^-)\in\dot{H}^1(\Omega^+)\times\dot{H}^1(\Omega^-):\,\,\lb f\rb=\eta(x)\lb\rho\rb\}
\end{split}
\end{equation*}
and the energy is given by
\begin{equation*}
\mathbb{E}[f^+,f^-]:=\mu^+\int_{\Omega^+}\vert \nabla_{x,y} f^+(x,y,t)\vert^2dxdy+\mu^-\int_{\Omega^-}\vert \nabla_{x,y} f^-(x,y,t)\vert^2dxdy.
\end{equation*}

\begin{lemma}\label{StructureOfA}

Assume $({\bf H}1)$ and $({\bf H}2)$. The affine set $\mathcal{A}$ is well defined, non-empty, and every bounded sequence in $\mathcal{A}$ for the natural semi-norm has a subsequence that converges weakly.

\end{lemma}

\begin{proof}[Proof of Lemma \ref{StructureOfA}]

Define $f^+_0(x,y)=0$ and $f^-_0(x,y)=\lb\rho\rb \eta(x)\chi(r^{-1}(y-\eta(x)))\mathfrak{1}_{\Omega^-}$, where $0<r<\hbox{dist}(\Sigma,\Gamma^-)$ and $\chi\in C^\infty_c(-1,1)$ is such that $\chi\equiv 1$ in a neighborhood of $0$. We have that $(f^+_0,f^-_0)\in\mathcal{A}$, which is then non-empty. It is proved in \cite[Proposition 3.3.]{NgPa} that $\dot{H}^1(\Omega^\pm)$ is a Hilbert space. Using \cite[Theorem 5.1]{LeoTice} or \cite[Theorem A.1.]{NgPa}, we see that $Tr^\pm: f^\pm\mapsto f^\pm(x,\eta(x))$ is well defined and continuous from $\dot{H}^1(\Omega^\pm)$ into $\widetilde{H}^\frac{1}{2}_\Theta(\mathbb{R}^d)$. Thus we see that $\mathcal{A}\subset \dot{H}^1(\Omega^+)\times\dot{H}^1(\Omega^-)$ is a closed affine subset of a Hilbert space.

%

\end{proof}

\begin{lemma}\label{UniquenessMinimizer}

Assuming $({\bf H}1)$ and $({\bf H}2)$, there exists a unique minimizer of \eqref{MinPb}, $(q^+,q^-)$, and in addition, $q^\pm$ satisfies \eqref{HarmonicProblem}.

\end{lemma}

\begin{proof}[Proof of Lemma \ref{UniquenessMinimizer}]

We refer to \cite[Section 3.1.]{NgPa} for similar computations. Since $\mathcal{A}$ is affine and $\mathbb{E}$ is strictly convex, uniqueness is direct. Considering a subsequence and using Lemma \ref{StructureOfA}, we obtain the existence. Harmonicity inside the domain follows by considering variations supported inside $\Omega^\pm$, and the Neumann conditions at $\Gamma^\pm$ and at $\Sigma$ follow similarly by considering perturbations supported near the top and bottom, and the interface, respectively. 

\end{proof}

\subsection{Comparison principle and Proof of the main Theorem}

\begin{proof}[Proof of Theorem \ref{MainThm}]
Once we have established that the velocity field $u$ comes from a minimizer, it suffices to find competitors in $\mathcal{A}$. A natural example is given by
\begin{equation}\label{Competitor}
\begin{split}
f^+(x,y):=\lb \rho\rb\max\{-y,0\},\qquad f^-(x,y):=\lb\rho\rb\max\{0,y\}.
\end{split}
\end{equation}
The competitor $(f^+,f^-)$ is Lipshitz and therefore admissible and we conclude that
\begin{equation*}
\begin{split}
\mathcal{E}<\mathbb{E}[f^+,f^-]=\lb\rho\rb^2\left[\mu^+\Vert \eta_-\Vert_{L^1}+\mu^-\Vert \eta_+\Vert_{L^1}\right]
\end{split}
\end{equation*}
which gives \eqref{AlternativeBound}. The inequality above is strict by uniqueness of the minimizer since $(f^+,f^-)$ are not harmonic. The optimality of the bound \eqref{AlternativeBound} follows from Lemma \ref{OptimalConstantLem}.

\end{proof}

\subsection{Extension to the one-fluid case}

In the one-fluid setting, the same argument gives a variational interpretation for $\mathcal{E}$. However, in this case, the competitor \eqref{Competitor} only works for a wave of elevation: $\eta\ge0$ (or a depression wave $\eta\le 0$). In particular, this gives 
\begin{equation*}
\begin{split}
\mathcal{E}_{1f}<\lb\rho\rb^2\mu\int_{\mathbb{R}^{d}} \vert\eta\vert dx.
\end{split}
\end{equation*}
In fact, this can also be obtained from the observation that\footnote{We thank Huy Nguyen for this observation.} \cite[Lemma 4.2.]{NgPa} $G(\eta)\eta<1$.

\subsection{Optimality}\label{Optimality}

One may wonder whether \eqref{AlternativeBound} gives a bound which is optimal in any way. Simple scaling arguments show that it can be saturated, at least up to a multiplicative constant.

\begin{lemma}\label{OptimalConstantLem}

Assume $({\bf H}1)$ and $({\bf H}2)$. There exists a sequence of smooth, compactly supported functions $\eta_n$ for which the inequality in \eqref{AlternativeBound} is saturated in the sense that:
\begin{equation*}
\mathcal{E}[\eta_n]/(\mu^+\Vert (\eta_n)_-\Vert_{L^1}+\mu^-\Vert(\eta_n)_+\Vert_{L^1})\to \lb\rho\rb^2.
\end{equation*}

\end{lemma}

\begin{remark}
	It is interesting to note that one can construct scenarios for which  $\mathcal{E}<\infty$ even though $\|\eta\|_{L^1}=\infty$.
\end{remark}

\begin{proof}[Proof of Lemma \ref{OptimalConstantLem}]
We will prove this in the case of a surface of elevation: $\eta\ge0$. The case $\eta\le 0$ follows similarly. In this case, we will only consider one domain and we define the Dirichlet energy of a function $\eta$ on a domain $\Omega$ to be $\mathcal{D}_\Omega(\eta):=\int_\Omega\vert\nabla\eta\vert^2dxdy$.

We first study the model case of a step function $\eta=\sigma:=H\cdot\mathfrak{1}_{\{0\le x\le L\}}$, for some constant height $H$, when the computations can be carried out explicitly and extend by continuity of the Dirichlet energy under deformation of the domain. The arguments are elementary but we give the details for the sake of completeness. We consider the rectangle $R:=\Omega^-\cap\{y\ge0\}=\{0\le x\le L,\, 0\le y\le H\}$, and we also partition the boundary into $\partial R=\mathcal{T}\cup\mathcal{B}$ with ``free'' boundary at the bottom $\mathcal{B}:=\{0\le x\le L,\,\, y=0\}$. Here $\mathcal{T}$ denotes the union of the top with the two vertical sides of the rectangle.

Let $u_R$ be the harmonic function satisfying the Dirichlet boundary condition $u_R=y$ on $\mathcal{T}$ and the Neumann condition $\partial_yu_R=0$ on $\mathcal{B}$. By standard arguments, $u_R$ is the minimizer of the Dirichlet energy with Dirichlet condition on $\mathcal{T}$:
\begin{equation*}
\begin{split}
\mathcal{D}_R(u_R):=\min\{\mathcal{D}_R(g):\,\, g(x,y)=y\hbox{ on }\mathcal{T}\}
\end{split}
\end{equation*}
and since $f^-$ is admissible, we see that
\begin{equation*}
\begin{split}
\mathcal{D}_R(u_R)\le\mathcal{D}_R(f^-):=\iint_R\vert\nabla f^-(x,y)\vert^2dxdy\le \iint_{\Omega^-}\vert\nabla f^-(x,y)\vert^2dxdy.
\end{split}
\end{equation*}

The energy of $u_R$ can be computed: letting $v=y-u_R$, we see that it satisfies $0$ Dirichlet boundary condition on top and sides and the Neumann condition $\partial_y v=1$ at the bottom. We can then expand into Fourier series to get
\begin{equation*}
\begin{split}
v(x,y):=\sum_{n\ge1} a_n\sin(\frac{\pi nx}{L})\sinh(\frac{n\pi}{L}(y-H)),\qquad 1=\sum_{n\ge 1}\frac{n\pi}{L}\cosh(\frac{n\pi}{L}H)a_n\sin(\frac{\pi nx}{L}).
\end{split}
\end{equation*}
The Fourier coefficients can easily be computed to be
\begin{equation*}
\begin{split}
a_n=\frac{2(1+(-1)^n)}{n\pi}\frac{L}{n\pi}\frac{1}{\cosh(\frac{n\pi}{L}H)}
\end{split}
\end{equation*}
and we conclude that
\begin{equation*}
\begin{split}
\Vert \nabla_{x,y}v\Vert_{L^2(R)}^2=\sum_{n\ge1} \frac{n^2\pi^2}{2L}Ha_n^2=(LH)\sum_{n\ge1} \frac{2}{n^2\pi^2}\frac{1}{\cosh^2(\frac{2n\pi}{L}H)}\le LH\cdot\frac{1}{3}\frac{1}{\cosh^2(\frac{2\pi}{L}H)},
\end{split}
\end{equation*}
where in the last step we used the fact that the maximum of $\big(\cosh^2(\tfrac{2n\pi}{L}H)\big)^{-1}$ is achieved at $n=1$. Moreover, notice that as $H\to\infty$, $\Vert \nabla_{x,y}v\Vert_{L^2(R)}^2\to0$.
As a result, we see that
\begin{equation*}
\begin{split}
\mathcal{E}&> \mu^-\lb\rho\rb^2\Vert \nabla_{x,y}u\Vert_{L^2(R)}^2= \mu^-\lb\rho\rb^2\Vert\nabla_{x,y}(y-v)\Vert_{L^2(R)}^2\ge  \mu^-\lb\rho\rb^2 LH(1+o_{H\to\infty}(1))\\
&>  \mu^-\lb\rho\rb^2\Vert \eta_+\Vert_{L^1}(1+o(1)).
\end{split}
\end{equation*}
It remains to modify $\eta$ to have a smooth function. This can be done by choosing $\eta$ as a smooth function bounded by two step functions of same height but slightly different support: $\sigma_1\le\eta\le\sigma_2$ with $\Vert\sigma_2-\sigma_1\Vert_{L^1}\ll \Vert \sigma_1\Vert_{L^1}$. We let $R_1$ be the rectangle associated to $\sigma_1$ and $R_2$ the rectangle associated to $\sigma_2$ and $\Omega^-_+:=\Omega^-\cap\{y\ge0\}$ so that $R_1\subset\Omega^-_+\subset R_2$ and $\hbox{Area}(R_2\setminus R_1)\ll \hbox{Area}(\Omega^-_+)$. We call $u_{R_1}$, $u_{R_2}$ the minimizers of the Dirichlet energy as before, and given a function $u$ defined on one of the above domains, we define $u^{ext}$ its extension to a bigger domain by $y$ (this is a Lipshitz function). Then, we see that, 
\begin{equation*}
\begin{split}
\mathcal{D}_{R_2}(u_{R_2})&<\mathcal{D}_{R_2}((f^-)^{ext})=\mathcal{D}_{\Omega^-_+}(\eta)+\Vert \sigma_2-\eta\Vert_{L^1},\\\
\mathcal{D}_{\Omega^-_+}(\eta)&<\mathcal{D}_{\Omega^-_+}(u_{R_1}^{ext})=\mathcal{D}_{R_1}(u_{R_1})+\Vert \eta-\sigma_1\Vert_{L^1},
\end{split}
\end{equation*}
and in particular
\begin{equation*}
\begin{split}
\vert\mathcal{D}_{\Omega^-_+}(\eta)-\Vert \eta\Vert_{L^1}\vert\le 3\Vert \sigma_2-\sigma_1\Vert_{L^1}\ll\Vert \eta\Vert_{L^1}.
\end{split}
\end{equation*}
%
This finishes the proof of Lemma \ref{OptimalConstantLem}.
\end{proof}

\subsection*{{\bf Acknowledgments}}

The authors would like to thank Prof. Gomez-Serrano for many stimulating discussions and constructive comments.

This material is based upon work supported by the National Science Foundation under Grant No. DMS-1929284 while SVH and BP were in residence at the Institute for Computational and Experimental Research in Mathematics in Providence, RI, during the program ``Hamiltonian Methods in Dispersive and Wave Evolution Equations''. The work of S.V.H is funded by the National Science Foundation through the award DMS-2102961. B.P. was also partially supported by a Simons fellowship and by NSF Grant No. DMS-2154162.

\end{document}